\documentclass[11pt]{article}
\usepackage{amsfonts,amsthm,amsmath}
\usepackage{amssymb}
\usepackage{authblk}
\usepackage[top=1.25in, bottom=1.5in, left=1.5in, right=1.5in]{geometry}

\newtheorem{Theorem}{Theorem}[section]
\newtheorem{Definition}{Definition}[section]
\newtheorem{Example}{Example}[section]

\newtheorem{Lemma}[Theorem]{Lemma}
\newtheorem{Corollary}[Theorem]{Corollary}

\renewcommand{\Pr}{{\ensuremath{\mathrm{Pr}}}}
\newcommand{\xx}{{\ensuremath{\mathbf{x}}}} 
\newcommand{\yy}{{\ensuremath{\mathbf{y}}}} 
\newcommand{\zz}{{\ensuremath{\mathbf{z}}}} 
\newcommand{\uu}{{\ensuremath{\mathbf{u}}}} 
\newcommand{\vv}{{\ensuremath{\mathbf{v}}}} 
\newcommand{\rr}{{\ensuremath{\mathbf{r}}}} 
\newcommand{\ww}{{\ensuremath{\mathbf{w}}}} 
\newcommand{\dist}{{\ensuremath{\mathsf{d}}}} 

\title{On Security Properties of All-or-nothing Transforms}
\author{Navid Nasr Esfahani}
\author{Douglas R.\ Stinson\thanks{D.R.\ Stinson's research is supported by  NSERC discovery grant RGPIN-03882.
}}
\affil{David R.\ Cheriton School of Computer Science, University of Waterloo,
Waterloo, Ontario, N2L 3G1, Canada\\ {\tt dstinson@uwaterloo.ca}}

\date{}

\begin{document}
	
\maketitle

\begin{abstract}
All-or-nothing transforms have been defined as bijective mappings on all $s$-tuples over a specified finite alphabet. These mappings are required to  satisfy certain ``perfect security'' conditions specified using entropies of the probability distribution defined on the input $s$-tuples. Alternatively, purely combinatorial definitions of AONTs have been given, which involve certain kinds of ``unbiased arrays''. However, the combinatorial definition makes no reference to probability definitions. 

In this paper, we examine the security provided by AONTs that satisfy the combinatorial definition. The security of the AONT can depend on the underlying probability distribution of the $s$-tuples. We show that perfect security is obtained from an AONT if and only if the input $s$-tuples are equiprobable.  However, in the case where  the input $s$-tuples are not equiprobable, we still achieve a weaker security guarantee. We also consider the use of randomized AONTs to provide perfect security for a smaller number of inputs, even when those inputs are not equiprobable. 
\end{abstract}

\textbf{Keywords:} all-or-nothing transform, perfect security

\section{Introduction}

All-or-nothing-transforms (AONTs) were invented in 1997 by Rivest \cite{R}. Several variations of AONTs have received considerable attention since then. Some early papers include
\cite{Boyko,CDHKS,Desai}. In this paper, we focus on unconditionally secure AONTs, which were introduced by Stinson \cite{St} 
and later generalized in \cite{DES,EGS}.
Further work focussing on the existence of  unconditionally secure AONTs can be found in \cite{ES,WCJ,ZZWG}.

AONTs were originally suggested by Rivest \cite{R} as a mode of operation for block ciphers that would slow down exhaustive key searches. There have since been numerous suggested applications of AONTs in security and cryptography. We do not survey these applications here; however, a variety of applications are discussed and reviewed in \cite{Phd}.

We begin with the informal definition of an unconditionally secure all-or-nothing-transform that was given in 
\cite{St}.

\begin{Definition}
\label{def1}
Suppose $\phi : \Gamma^s \rightarrow \Gamma^s$, where $\Gamma$ is a finite set of size $v$ (called an \emph{alphabet}) and $s$ is a positive integer. Thus $\phi$ is a function that maps an input $s$-tuple 
$\mathbf{x} = (x_1, \dots  , x_s)$ to
output $s$-tuple 
$\mathbf{y} = (y_1, \dots  , y_s)$. 
The function $\phi$ is an \emph{$(s,v)$-all-or-nothing transform} (or $(s,v)$-AONT) provided that the following properties are
satisfied:
\begin{enumerate}
\item  $\phi$ is a bijection.
\item  If any $s - 1$ of the $s$ outputs $y_1, \dots , y_s$ are fixed, then the value of any one input  $x_i$ (for $1 \leq i \leq s$) is completely undetermined.
\end{enumerate}
\end{Definition}

Definition \ref{def1} does not place any bound on the computational capabilities of an adversary. 
In contrast, Rivest's original definition from \cite{R} only requires that the computation of any one input value, given $s-1$ output values, is infeasible. 

Definition \ref{def1} was generalized to that of a $t$-all-or-nothing-transform in \cite{DES}.
\begin{Definition}
\label{deft}
Suppose $1 \leq t \leq s$. A function $\phi : \Gamma^s \rightarrow \Gamma^s$ is a \emph{$(t,s,v)$-all-or-nothing transform} (or $(t,s,v)$-AONT) provided that the following properties are
satisfied:
\begin{enumerate}
\item  $\phi$ is a bijection.
\item  If any $s - t$ of the $s$ outputs $y_1, \dots , y_s$ are fixed, then the values of any $t$ inputs $x_i$ (for $1 \leq i \leq s$) are completely undetermined.
\end{enumerate}
\end{Definition}
We note that Definition \ref{def1} is  just the special case of Definition \ref{deft} that arises by setting  $t=1$. 

Definition \ref{deft} was ``rephrased'' in terms of the entropy function in \cite{DES}, as follows. (Earlier, an analogous definition was given in \cite{St} for the special case $t=1$.)
We will refer to Definition \ref{defentropy} as the \emph{entropy definition}.

\begin{Definition}
	\label{defentropy}
	Let 
	\[\mathbf{X_1}, \dots , \mathbf{X_s}, \mathbf{Y_1}, \dots , \mathbf{Y_s}\] be random variables taking on
	values in the finite set $\Sigma$ of size $v$.  
	These $2s$ random variables 
	define a \emph{$(t, s, v)$-AONT} provided that
	the following conditions are satisfied:
	\begin{enumerate}
		\item $\mathsf{H}( \mathbf{Y_1}, \dots , \mathbf{Y_s} \mid \mathbf{X_1}, \dots , \mathbf{X_s}) = 0$.
		\item $\mathsf{H}( \mathbf{X_1}, \dots , \mathbf{X_s} \mid \mathbf{Y_1}, \dots , \mathbf{Y_s}) = 0$.
		\item For all $\mathcal{X} \subseteq \{\mathbf{X_1}, \dots , \mathbf{X_s}\}$ with 
		$|\mathcal{X}|  = t$, and for all $\mathcal{Y} \subseteq \{\mathbf{Y_1}, \dots , \mathbf{Y_s}\}$ with 
		$|\mathcal{Y}|  = s-t$, it holds that 
		\begin{equation}
			\label{t-AONT_PS.eq}
			\mathsf{H}( \mathcal{X}  \mid  \mathcal{Y} ) = \mathsf{H}(\mathcal{X}).
		\end{equation}
	\end{enumerate}
\end{Definition}
Suppose we consider $\mathbf{X_1}, \dots , \mathbf{X_s}$ as inputs and $\mathbf{Y_1}, \dots , \mathbf{Y_s}$ as outputs. 
Then properties 1  and 2  ensure that we can define a bijection between the $s$ inputs and the $s$ outputs. Property 3  is saying that no information about any $t$ inputs can be derived from any $s-t$ outputs. 

Finally, a combinatorial definition of all-or-nothing definitions was proposed in \cite{St} in the case $t=1$ and in \cite{DES} for arbitrary $t$. First, we require some preliminary definitions. An \emph{$(N,k,v)$-array} is an $N$ by $k$ array, say $A$, whose entries are elements chosen from an alphabet $\Gamma$ of order $v$.  
Suppose the columns of $A$ are
labeled by the elements in the set $C$.  
Let $D \subseteq C$, and define $A_D$ to be the array obtained from $A$
by deleting all the columns $c \notin D$.
We say that $A$ is
\emph{unbiased} with respect to $D$ if the rows of
$A_D$ contain every $|D|$-tuple of elements of $\Gamma$ 
exactly $N / v^{|D|}$ times.

We have already stated that a $(t,s,v)$-AONT, say $\phi$,  is a bijection from $\Gamma$ to $\Gamma$, where $\Gamma$ is a $v$-set.
The \emph{array representation} of $\phi$ is a  $(v^s,2s,v)$-array, say $A$, that is constructed as follows.
For every input $s$-tuple $(x_1, \dots , x_s) \in \Gamma^s$, there is a row of $A$ containing the entries $x_1, \dots , x_s, y_1, \dots , y_s$, 
where $\phi(x_1, \dots , x_s) = (y_1, \dots , x_y)$. 

Definition \ref{defunbiased} defines $(t,s,v)$-AONT in terms of
arrays that are unbiased with respect to
certain subsets of columns. 
We refer to this definition as the \emph{combinatorial definition}.

\begin{Definition}
\label{defunbiased}
A \emph{$(t,s,v)$-all-or-nothing transform} is a $(v^s,2s,v)$-array, say $A$, with columns 
labeled $1, \dots , 2s$, 
that is unbiased with respect to the following subsets of columns:
\begin{enumerate}
\item $\{1, \dots , s\}$,
\item $\{s+1, \dots , 2s\}$, and
\item $I \cup J$, 
         for all $I \subseteq \{1,\dots , s\}$ with $|I| = t$ and all 
         $J \subseteq \{s+1,\dots , 2s\}$ with $|J| = s-t$.
\end{enumerate}
\end{Definition}

We interpret the first $s$ columns of $A$ as indexing the $s$ inputs and the last $s$ columns as indexing the $s$ outputs.
Then, as mentioned above, properties 1 and 2 ensure that the array $A$ defines a bijection $\phi$. Property 3 
says that knowledge of any $s-t$ outputs does not rule out any possible values for any $t$ inputs.

\subsection{Our Contributions}

Our goal in this paper is to better understand the definitions of AONTs given above and analyze the differences between them. 
The entropy definition (Definition \ref{defentropy}) involves the ``security'' of an AONT, while the combinatorial definition (Definition \ref{defunbiased}) is just defining a certain mathematical structure. An analysis of the security properties of AONTs will, in general, depend on the underlying probability distribution on the possible inputs. This dependence has not been discussed in prior work.

It turns out to be illuminating to also consider a security definition that is not a stringent 
as Definition \ref{defentropy}. We call this \emph{weak security} (see Definition \ref{Def:tAONT_WS}), in contrast to the security afforded in Definition \ref{defentropy}, which we call \emph{perfect security}.

Our two main results are 
\begin{enumerate}
\item Any AONT satisfying Definition \ref{defunbiased} (the combinatorial definition) is guaranteed to provide weak security.
\item An AONT satisfying Definition \ref{defunbiased} provides perfect security if and only if the underlying probability distribution on the input $s$-tuples is uniform.
\end{enumerate}

We also show that we can obtain perfect security for $t$ inputs, for an arbitrary probability distribution, by using a randomized AONT.


\section{Perfect and Weak Security of AONTs}

In the rest of this paper, we  assume that every input $s$-tuple occurs with non-zero probability.  Since an AONT is a bijection, it follows immediately that every output $s$-tuple also occurs with non-zero probability. 

If Definition \ref{defentropy} is satisfied, then 
the probability that $t$ inputs take on any $t$ specified values, given the values of any $s-t$ outputs, is the same as the \textit{a priori} probability that they take on the same values. We call this \emph{perfect security}. We will prove in Theorem \ref{Thrm:unbiased_PS} that Condition \ref{t-AONT_PS.eq}  of Definition \ref{defentropy} can be satisfied if the input $s$-tuples all occur with uniform probability.  

We also consider a notion that we call \emph{weak security}, 
where we require that any $t$ inputs can take on any  $t$ specified values with non-zero probability, given the values of any $s-t$ outputs. More formally, we have the following entropy definition for a weakly secure AONT.

\begin{Definition}
	\label{Def:tAONT_WS}
	Let 
	\[\mathbf{X_1}, \dots , \mathbf{X_s}, \mathbf{Y_1}, \dots , \mathbf{Y_s}\] be random variables taking on
	values in the finite set $\Sigma$ of size $v$.  
	These $2s$ random variables 
	define a \emph{weakly secure $(t, s, v)$-AONT} provided that
	the following conditions are satisfied:
	\begin{enumerate}
		\item $\mathsf{H}( \mathbf{Y_1}, \dots , \mathbf{Y_s} \mid \mathbf{X_1}, \dots , \mathbf{X_s}) = 0$.
		\item $\mathsf{H}( \mathbf{X_1}, \dots , \mathbf{X_s} \mid \mathbf{Y_1}, \dots , \mathbf{Y_s}) = 0$.
		\item \label{t-AONT_WS.eq} Given the values of any $s-t$ outputs, any $t$ inputs take on any possible values with a non-zero probability.
	\end{enumerate}
\end{Definition}

It is immediate that a perfectly secure AONT is also weakly secure. We illustrate the concepts of perfect and weak security in the following examples.

\begin{table}[t]
		\begin{center}
			\caption{A $(1,2,2)$-AONT over the alphabet $\{a,b,c\}$}
			\label{Tab:122AONT}
			\vspace{.25in}
			\begin{tabular}{|c|c||c|c|} \hline
				$x_1$ & $x_2$ & $y_1$ & $y_2$ \\ \hline
				a & a & a & a \\
				a & b & c & b \\
				a & c & b & c \\
				b & a & b & b \\
				b & b & a & c \\
				b & c & c & a \\
				c & a & c & c \\
				c & b & b & a \\
				c & c & a & b \\\hline				
			\end{tabular}
		\end{center}
	\end{table}

\begin{Example}
	\label{2_AONT_Example}
{\rm	
	Table \ref{Tab:122AONT} presents the array representation of a $(1,2,2)$-AONT, over the alphabet $\Gamma = \{a,b,c\}$. In the rows of this array, we are just listing the outputs $y_1$ and $y_2$ corresponding to  all possible values of the input elements $x_1$ and $x_2$. 
	
	Suppose that all nine input pairs are equally probable, and 
	suppose an adversary learns that $y_2 = a$. Then each possible value of $x_1$ occurs with the same probability. That is, 
	\[ \Pr [\mathbf{X_1} = a \mid \mathbf{Y_2} = a] = \Pr [\mathbf{X_1} = b \mid \mathbf{Y_2} = a] = \Pr [\mathbf{X_1} = c \mid \mathbf{Y_2} = a] = 
	\frac{1}{3}.\]
	Since
	\[ \Pr [\mathbf{X_1} = a ] = \Pr [\mathbf{X_1} = b ] = \Pr [\mathbf{X_1} = c ] = 
	\frac{1}{3},\]
	we have 
	\[ \Pr [\mathbf{X_1} = a \mid \mathbf{Y_2} = a] = \Pr [\mathbf{X_1} = a ],\]
	etc.
	Using similar calculations, it  can be verified that this AONT provides perfect security for an equiprobable input distribution. \hfill$\blacksquare$
	
	}

\end{Example}

\begin{Example}
	\label{2_AONT_Example-2}
{\rm		
	Now we consider a nonuniform input distribution for the AONT presented in Table \ref{Tab:122AONT}. Suppose that the inputs $x_1$ and $x_2$ are independent, and they occur with the following probabilities:

	\[
	\begin{array}{l@{\quad\quad}l@{\quad\quad}l}
	\Pr [\mathbf{X_1} = a ] = 1/3 & \Pr [\mathbf{X_1} = b ] = 1/3 & 
	\Pr [\mathbf{X_1} = c] = 1/3 \\
	\Pr [\mathbf{X_2} = a ] = 1/2 & \Pr [ \mathbf{X_2} = b] = 1/4 & 
	\Pr [\mathbf{X_2} = c] = 1/4 .
	\end{array}
	\]
	Again, suppose an adversary learns that $y_2 = a$. 
	We  compute the conditional probability distribution on $x_1$, given that $y_2 = a$. 
	First we note that
	\begin{eqnarray*}
	 \Pr [\mathbf{Y_2} = a] &=& \Pr [\mathbf{X_1} = a, \mathbf{X_2} = a] + \Pr [\mathbf{X_1} = b, \mathbf{X_2} = c] 
	               + \Pr [\mathbf{X_1} = c, \mathbf{X_2} = b] \\
	 &=& \frac{1}{3} \times \frac{1}{2} + \frac{1}{3} \times \frac{1}{4} 
	 + \frac{1}{3} \times \frac{1}{4} \\
	 & = & \frac{1}{6}  + \frac{1}{12} 
	 + \frac{1}{12} \\ &=&
	 \frac{1}{3}.
	\end{eqnarray*}
	Now, we have
	 \begin{eqnarray*}
	 \Pr [\mathbf{X_1} = a \mid \mathbf{Y_2} = a] & = & \frac{\Pr [\mathbf{X_1} = a , \mathbf{Y_2} = a]}{\Pr [\mathbf{Y_2} = a]} \\
	 & = & \frac{\Pr [\mathbf{X_1} = a , \mathbf{X_2} = a]}{\Pr [\mathbf{Y_2} = a]} \\
	 & = & \frac{1/6}{1/3}\\
	 & = & \frac{1}{2}.
	 \end{eqnarray*}
	 Similar calculations yield 
	 \begin{eqnarray*}
	 \Pr [\mathbf{X_1} = b \mid \mathbf{Y_2} = a] & = &  \frac{1}{4}
	 \end{eqnarray*}
	 and
	 \begin{eqnarray*}
	 \Pr [\mathbf{X_1} = c \mid \mathbf{Y_2} = a] & = &  \frac{1}{4}.
	 \end{eqnarray*}
	 Thus, when $y_2 = a$,  the \emph{a posteriori} distribution on $x_1$ is different from  the \emph{a priori} distribution on $x_1$. This is sufficient to show that the AONT does not provide perfect security.	
	 
	 \smallskip
	 
	 It is interesting to repeat these calculations, considering the distributions on $x_2$ instead of $x_1$. We obtain
	 \begin{eqnarray*}
	 \Pr [\mathbf{X_2} = a \mid \mathbf{Y_2} = a] &=& \frac{1}{2}\\
	 \Pr [\mathbf{X_2} = b \mid \mathbf{Y_2} = a] &=& \frac{1}{4}\\
	 \Pr [\mathbf{X_2} = c \mid \mathbf{Y_2} = a] &=& \frac{1}{4}.	 
	 \end{eqnarray*}
	 Thus, when $y_2 = a$,  the \emph{a posteriori} distribution on $x_2$ is identical to  the \emph{a priori} distribution on $x_1$. A similar result holds when $y_2 = a$, and when $y_1 = b$ or $y_1 =  a$.
	 \hfill$\blacksquare$	
	}
\end{Example}

\begin{Theorem}
	\label{equiv}
	A weakly secure $(t,s,v)$-AONT is equivalent to a $(v^s,2s,v)$-array
	that is unbiased with respect to the following subsets of columns:
	\begin{enumerate}
		\item $\{1, \dots , s\}$,
		\item $\{s+1, \dots , 2s\}$, and
		\item $I \cup J$, 
		for all $I \subseteq \{1,\dots , s\}$ with $|I| = t$ and all 
		$J \subseteq \{s+1,\dots , 2s\}$ with $|J| = s-t$.
	\end{enumerate}
\end{Theorem}

\begin{proof}
	Let $A$ be the hypothesized $(v^s,2s,v)$-array on alphabet $\Gamma$, $|\Gamma| = v$. We construct $\phi : \Gamma^s \rightarrow \Gamma^s$ as follows: for each row $(x_1, \dots , x_{2s})$ of $A$, define
	\[ \phi (x_1, \dots , x_{s}) = (x_{s+1}, \dots , x_{2s}).\]
	Being unbiased with respect to the first two subsets of columns indicates that $\phi$ is a bijection, and being unbiased with respect to the third subset of columns is equivalent to condition 3 of Theorem \ref{Def:tAONT_WS}. Hence, the function $\phi$ is a weakly secure $(t,s,v)$-AONT. 
	
	Conversely, suppose $\phi$ is a weakly secure $(t,s,v)$-AONT.
	Let $A$ be the array representation of the AONT. 
	Then $A$ is the desired $(v^s,2s,v)$-array.
\end{proof}


Now we analyze perfect security. We make use of the following well-known fact.

\begin{Theorem}
\label{conditional}
Let $\mathcal{X}$ and $\mathcal{Y}$  be random variables. Then
$\mathsf{H}( \mathcal{X}  \mid  \mathcal{Y} ) = \mathsf{H}(\mathcal{X})$ if and only if
$\mathcal{X}$ and $\mathcal{Y}$ are independent. 
\end{Theorem}

Suppose $A$ is the array representation of a $(t,s,v)$-AONT,  
where $1 \leq t < s$.
From Theorem \ref{conditional}, we have perfect security if and only if any $t$ inputs are independent of any $s-t$ outputs.
For an input $s$-tuple $\xx$ and for any $I \subseteq \{1, \dots ,s\}$, 
let $\xx_I = (x_i : i \in I)$. Thus, $\xx_I$  is formed by taking the row in $A$ corresponding to the input $\xx$, restricted to the columns in $I$.
Similarly, for an  output $s$-tuple $\yy$ and for any $J \subseteq \{1, \dots ,s\}$, let $\yy_J = (y_j : j \in J)$. Therefore,  $\yy_J$  is obtained by taking the row in $A$ corresponding to the 
output $\yy$, restricted to the columns in $J$.

We let $\mathbf{X}$ be a random variable that denotes an input $s$-tuple, and $\mathbf{Y}$ is a random variable that denotes an output $s$-tuple. $\mathbf{X_I}$ and $\mathbf{Y_J}$ are the random variables induced by specified subsets of the $s$ (respective) co-ordinates. 

Then the perfect security condition can be written as follows:
\begin{equation}
\label{perfect.eq}
\Pr [\mathbf{X_I} = \uu, \mathbf{Y_J}= \vv] = \Pr [\mathbf{X_I}= \uu] \, \Pr [ \mathbf{Y_J}= \vv]
\end{equation}
for all $|I| = t$ and $|J| = s-t$, and for all $t$-tuples $\uu$ and all $(s-t)$-tuples $\vv$.


\begin{Theorem}\label{Thrm:unbiased_PS}
	Suppose a $(t,s,v)$-AONT has an array representation, say $A$,  that satisfies Definition \ref{defunbiased}, and suppose that  all the input $s$-tuples are equally probable. Then the AONT is  perfectly secure.
\end{Theorem}

\begin{proof}
	We prove this theorem by showing that values of any $t$ inputs are independent of any $s-t$ outputs. 
	
	In equation (\ref{perfect.eq}), we compute $\Pr [\mathbf{X_I} = \uu,\mathbf{Y_J}= \vv]$, $\Pr [\mathbf{X_I} = \uu]$ and $\Pr [\mathbf{Y_J}= \vv]$ as follows.
Suppose we fix $I$, $J$, $\uu$ and $\vv$, where $|I| = t$, $|J| = s-t$, $\uu$ is a $t$-tuple and $\vv$ is an $(s-t)$-tuple.
There is exactly one input $s$-tuple, say $\zz$, such that 
$\zz_I = \uu$ and $\phi(\zz)_J = \vv$.  Then 
\[\Pr [\mathbf{X_I} = \uu,\mathbf{Y_J}= \vv] = \Pr[\zz].\]
We also have
\[\Pr [\mathbf{X_I} = \uu] = \sum _{\{ \xx : \xx_I = \uu\}} \Pr[\xx]\]
and
\[\Pr [\mathbf{Y_J}= \vv] = \sum _{\{ \xx : \phi(\xx)_J = \vv \}} \Pr[\xx].\]

For a fixed $I$, and given any $t$-tuple $\uu$, there are $v^{s-t}$ rows of $A$ such that 
	$\xx_I = \uu$. Since all  $v^s$ rows  of $A$ equiprobable, the probability of any input $t$-tuple taking value $\uu$ can be calculated as follows:
	\[\Pr [\mathbf{X_I} = \uu] = \frac{v^{s-t}}{v^s} = v^{-t}.\]
	Similarly, for a fixed $J$, and given any $(s-t)$-tuple $\vv$, there are $v^t$ rows of $A$ such that 
	$\yy_J = \vv$. Hence, the probability of any specified output $(s-t)$-tuple is
	\[\Pr [\mathbf{Y_J}= \vv] = \frac{v^t}{v^{s}}= v^{t-s}.\]
	We also know that any input $t$-tuple and output $(s-t)$-tuple  appear together in exactly one row of $A$. Thus,
	\[\Pr[\mathbf{X_I} = \uu, \mathbf{Y_J}= \vv]= \frac{1}{v^s}.\]
	Hence, for any given input $t$-tuple and any given output $(s-t)$-tuple, we have 
	\[\Pr[\mathbf{X_I} = \uu, \mathbf{Y_J}= \vv]= v^{-s} = v^{-t} v^{t-s}= \Pr[\mathbf{X_I} = \uu] \, \Pr[\mathbf{Y_J}= \vv].\]
	Therefore, any specified $t$ inputs and any specified $s-t$ outputs are independent.
\end{proof}

We have proved  that $A$ provides perfect security if the probability 
distribution defined on the input $s$-tuples is equiprobable. Now we prove the converse.
%
%
Assume we have perfect security. 
Suppose we fix an input $s$-tuple $\xx$ and we also fix $I$ such that $|I| = t$.
Denote $\yy = \phi(\xx)$ and $\xx_I = \uu$.
For any $J$ with $|J| = s-t$, 
from equation (\ref{perfect.eq}), we have
\[  \Pr [\xx] = \Pr [\mathbf{X_I} = \uu]  \, \Pr [\mathbf{Y_J} = \yy_J].\]
Since $t > 0$, we can choose a $J' \neq J$ such that $|J'| = s-t$.
Then
\[  \Pr [\xx] = \Pr [\mathbf{X_I} = \uu]  \, \Pr [\mathbf{Y_{J'}} = \yy_{J'}].\]
Since $\xx$, $I$ and $\uu$ are fixed and since $\Pr [\mathbf{X_I} = \uu] \neq 0$, the two previous equations imply that
\begin{equation}
\label{eq2}
  \Pr [\mathbf{Y_J} = \yy_J] = \Pr [\mathbf{Y_{J'}} = \yy_{J'}]
  \end{equation} for all $J, J'$ with $J \neq J'$ and $|J| = |J'| = s-t$.

Suppose that $\vv$ and  $\ww$ are $(s-t)$-tuples and $J$ and $J'$ are fixed, where $J \neq J'$. 
We say that the pair $(\vv,\ww)$ is \emph{$(J,J')$-compatible} if there is
an $s$-tuple $\yy$ such that $\yy_J = \vv$ and $\yy_{J'} = \ww$. Equivalently,
$\vv$ and $\ww$ are $(J,J')$-compatible if they agree on all co-ordinates in $J \cap J'$. 
The following lemma is a  consequence of (\ref{eq2}).
\begin{Lemma}
\label{lem1}
Suppose that $\vv$ and  $\ww$ are $(s-t)$-tuples and $J$ and $J'$ are fixed, where $J \neq J'$. If $(\vv,\ww)$ is $(J,J')$-compatible, then 
 \[\Pr [\mathbf{Y_J} = \vv] = \Pr [\mathbf{Y_{J'}} = \ww ] .\]
\end{Lemma}
Suppose that $(\vv,\ww)$ and $(\vv',\ww)$ are both $(J,J')$-compatible.
From Lemma \ref{lem1}, it follows that
\[\Pr [\mathbf{Y_J} = \vv] = \Pr [\mathbf{Y_{J'}} = \ww ] \]
and \[\Pr [\mathbf{Y_J} = \vv'] = \Pr [\mathbf{Y_{J'}} = \ww ] ,\] so
\begin{equation}
\label{eq3}
\Pr [\mathbf{Y_J} = \vv] = \Pr [\mathbf{Y_J} = \vv'].
\end{equation}

Let $\dist(\cdot,\cdot)$ denote the hamming distance between any two vectors of the same length. We have the following lemma.
\begin{Lemma}
\label{lem2}
Suppose that that  $|J| = s-t$, and suppose that $\vv$ and  $\vv'$ are $(s-t)$-tuples such that  $\dist(\vv,\vv') = 1$.
Then 
\[\Pr [\mathbf{Y_J} = \vv] = \Pr [\mathbf{Y_J} = \vv'].\]
\end{Lemma}
\begin{proof}
Choose $J'$ such that $|J \cap J'| = s-t-1$, and $(\vv,\ww)$ and $(\vv',\ww)$ are both $(J,J')$-compatible. That is $J'$ contains the $s-t-1$ co-ordinates where $\vv$ and $\vv'$ agree, along with one additional co-ordinate not in $J$, and $\ww$ agrees with $\vv$ and $\vv'$ on the $s-t-1$ common co-ordinates. Then the desired result follows from (\ref{eq3}).
\end{proof} 

\begin{Lemma}
\label{lem3}
For all $J$ such that $|J| = s-t$, and for any two distinct $(s-t)$-tuples $\vv$ and $\vv'$, it holds that 
\[\Pr [\mathbf{Y_J} = \vv] = \Pr [\mathbf{Y_J} = \vv'].\]
\end{Lemma}
\begin{proof}
We prove the result by induction on $\dist(\vv,\vv')$, where $1 \leq \dist(\vv,\vv') \leq s-t$.
Lemma \ref{lem2} establishes the base case, where  $\dist(\vv,\vv') = 1$.
Suppose the result holds when $\dist(\vv,\vv') \leq d$, where $1 \leq d \leq s-t-1$. It is easy to find $\vv''$ such that $\dist(\vv,\vv'') = d-1$ and $\dist(\vv'',\vv') = 1$.
By induction, we have 
\[\Pr [\mathbf{Y_J} = \vv] = \Pr [\mathbf{Y_J} = \vv'']\] and
\[\Pr [\mathbf{Y_J} = \vv''] = \Pr [\mathbf{Y_J} = \vv'],\] so it follows immediately that 
\[\Pr [\mathbf{Y_J} = \vv] = \Pr [\mathbf{Y_J} = \vv'].\]
\end{proof} 

\begin{Corollary}
\label{cor1}
$\Pr [\mathbf{Y_J} = \vv] = 1/v^{s-t}$ for all $(s-t)$-tuples $\vv$.
\end{Corollary}
\begin{proof}
There are $v^{s-t}$ choices for $\vv$, and $\Pr [\mathbf{Y_J} = \vv] $ is independent of $\vv$ from Lemma \ref{lem3}.
\end{proof}

By  similar arguments, we can obtain the following.
\begin{Lemma}
\label{lem4}
\[\Pr [\mathbf{X_I} = \uu] = \Pr [\mathbf{X_I} = \uu']\]
for all $I$ such that $|I| = t$, and for any two distinct $t$-tuples $\uu$ and $\uu'$.
\end{Lemma}

\begin{Corollary}
\label{cor2}
$\Pr [\mathbf{X_I} = \uu] = 1/v^{t}$ for all $t$-tuples $\uu$.
\end{Corollary}

\begin{Theorem} If a $(t,s,v)$-AONT provides perfect security, then 
$\Pr [\mathbf{X} = \xx] = 1/v^s$ for all $s$-tuples $\xx$.
\end{Theorem}
\begin{proof}
Choose any $s$-tuple $\xx$ and any $I$ and $J$ with $|I| = t$ and $|J| = s-t$. Let $\yy = \phi(\xx)$. 
Denote $\xx_I = \uu$ and $\yy_J= \vv$.
From (\ref{perfect.eq}), we have
\[\Pr [\mathbf{X} = \xx] = \Pr [\mathbf{X_I} = \uu, \mathbf{Y_J} = \vv] 
= \Pr [\mathbf{X_I} = \uu] \, \Pr [\mathbf{Y_J} = \vv].\]
Now, from Corollaries \ref{cor1} and \ref{cor2}, we have
\[\Pr [\mathbf{X_I} = \uu] = \frac{1}{v^{t}}\]
and 
\[\Pr [\mathbf{Y_J} = \vv] = \frac{1}{v^{s-t}}.\]
Therefore,
\[\Pr[\xx] = \frac{1}{v^{t}} \times \frac{1}{v^{s-t}} =  \frac{1}{v^{s}}.\]
\end{proof}

\section{Randomized AONTs}

Randomized AONTs were proposed by Rivest \cite{R} and they have since been considered by several authors. In this section, we show how a randomized AONT can provide perfect security when the inputs are drawn from an arbitrary probability distribution. Suppose we have a
weakly secure $(t,s,v)$-AONT, say $\phi : \Gamma^s \rightarrow \Gamma^s$. We use it to construct a \emph{randomized AONT} that transforms $t$ inputs into $s$ outputs, as described in Figure \ref{rand.fig}.

We will prove that the perfect security condition is satisfied for the $t$ designated inputs. But first, we observe that Example \ref{2_AONT_Example-2} provides an illustration. We can view 
Example \ref{2_AONT_Example-2} as a randomized AONT, where $x_2$ is the designated input and $x_1$ is as random input. We noted already that this example yields perfect security for the input $x_2$. 

Here is the statement and proof of the security of randomized AONT in general.

\begin{figure}
\caption{A Randomized AONT}
\label{rand.fig}

\begin{center}
\begin{tabular}{lp{5in}}
\textbf{input} & A $(t,s,v)$-AONT, say $\phi$, and $t$ inputs. 
\\
\textbf{step 1} & Assign the $t$ given inputs to any $t$ of the $x_i$'s (we call these inputs \emph{designated inputs}).\\
\textbf{step 2} & Choose the remaining $s-t$ of the $x_i$'s independently and uniformly at random from $\Gamma$ (we call these inputs \emph{random inputs}). \\
\textbf{step 3} & Output $(y_1, \dots , y_s) = \phi (x_1, \dots , x_s)$.
\end{tabular}
\end{center}
\end{figure}

\begin{Theorem}
	Suppose we use a weakly secure $(t,s,v)$-AONT as a randomized AONT, as described in Figure \ref{rand.fig}. Let 
	$\mathcal{X}$ denote the $t$ designated inputs. Then, for all $\mathcal{Y} \subseteq \{\mathbf{Y_1}, \dots , \mathbf{Y_s}\}$ with 
		$|\mathcal{Y}|  = s-t$, it holds that 
		\[
			\mathsf{H}( \mathcal{X}  \mid \mathcal{Y} ) = \mathsf{H}(\mathcal{X}).
			\]
\end{Theorem}

\begin{proof}
We assume an arbitrary probability distribution on $\mathcal{X}$, subject to the condition that
$\Pr [\mathcal{X} = \uu] > 0$ for all $\uu$. Let $\mathcal{R}$ denote the $s-t$ inputs that are chosen randomly; thus $\Pr [\mathcal{R} = \rr] = 1/v^{s-t}$ for all $(s-t)$-tuples $\rr$.

We show that $\mathcal{X}$ is independent of $\mathcal{Y}$. That is, we prove that the following equation holds:
\begin{equation}
\label{perfectr.eq}
\Pr [\mathcal{X} = \uu, \mathbf{Y_J}= \vv] = \Pr [\mathcal{X} = \uu] \, \Pr [ \mathbf{Y_J}= \vv]
\end{equation}
for all $|J| = s-t$, and for all $t$-tuples $\uu$ and all $(s-t)$-tuples $\vv$.

We first compute the probability distribution on $\mathbf{Y_J}$. Fix an $(s-t)$-tuple $\vv$.
Then, for any $t$-tuple $\uu$, there is a unique $(s-t)$-tuple $\rr$ such that $\phi(\uu, \rr)_J = \vv$. This is easily seen from the fact that the array representation of $\phi$ is unbiased with respect to the columns corresponding to the $t$ designated inputs and $J$. Thus we have
\begin{eqnarray}
\nonumber \Pr [ \mathbf{Y_J}= \vv] &=& \sum _{\uu} \left( \Pr [\mathcal{X} = \uu] \times \frac{1}{v^{s-t}} \right) \\
\nonumber &=& \frac{1}{v^{s-t}} \sum _{\uu} \Pr [\mathcal{X} = \uu] \\
\label{r1.eq}&=& \frac{1}{v^{s-t}}.
\end{eqnarray}
Now, we compute the joint probability distribution on $\mathcal{X} \times \mathbf{Y_J}$. 
As noted above, $\uu$ and $\vv$ uniquely determine $\rr$.  Therefore it is immediate that
\begin{eqnarray}\label{r2.eq}
\Pr [\mathcal{X} = \uu, \mathbf{Y_J}= \vv] 
&=&  \Pr [\mathcal{X} = \uu] \times \frac{1}{v^{s-t}} .
\end{eqnarray}
Finally, from (\ref{r1.eq}) and (\ref{r2.eq}),
we see that (\ref{perfectr.eq}) holds.
\end{proof}

\section{Summary}

We have proven that the combinatorial definition of an AONT provides perfect security only for 
an equiprobable distribution of the input $s$-tuples. In the case where we do not have an equiprobable input distribution, we could instead consider  the \emph{mutual information}
$I(\mathcal{X}  ; \mathcal{Y} ) =  \mathsf{H}(\mathcal{X}) - \mathsf{H}( \mathcal{X}  \mid  \mathcal{Y} ) $ for all relevant $\mathcal{X}$  and $\mathcal{Y}$.  It would be of interest to prove an upper bound on $I(\mathcal{X}  ; \mathcal{Y} )$, which would presumably depend on 
$\mathsf{H}(\mathcal{X})$.

\section*{Acknowledgements}

The authors would like thank Ian Goldberg for raising the issues we discuss in this paper.


\begin{thebibliography}{XX}

\bibitem{Boyko}
V.\ Boyko.
On the security properties of OAEP as an
all-or-nothing Transform.
{\em Lecture Notes in Computer Science}
{\bf 1666} (1999),
503--518 (CRYPTO '99).

\bibitem{CDHKS}
R.\ Canetti, Y.\ Dodis, S.\ Halevi, E.\  Kushilevitz and A. Sahai.
Exposure-resilient functions and all-or-nothing transforms.
{\em Lecture Notes in Computer Science}
{\bf 1807} (2000),
453--469 (EUROCRYPT 2000).

\bibitem{Desai}
A.\ Desai.
The security of all-or-nothing encryption:
protecting against exhaustive key search.
{\em Lecture Notes in Computer Science}
{\bf 1880} (2000),
359--375 (CRYPTO 2000).
 
\bibitem{CD}
C.J.\ Colbourn and J.H.\ Dinitz, eds.
{\em The CRC Handbook of Combinatorial Designs, Second Edition},
CRC Press, 2006.

 \bibitem{DES}
 P.\ D'Arco, N.\ Nasr Esfahani and D.R.\ Stinson.
 All or nothing at all. 
 {\em Electronic Journal of Combinatorics}  \textbf{23(4)} (2016), paper \#P4.10, 24 pp.
 
  \bibitem{Phd}
 N.\ Nasr Esfahani.
{\em Generalizations of all-or-nothing transforms
and their application in secure distributed
storage.}
PhD thesis, University of Waterloo, 2021.

 
 \bibitem{ES}
 N.\ Nasr Esfahani and D.R.\ Stinson.
Computational results on invertible matrices with the 
maximum number of invertible $2 \times 2$ submatrices.
\emph{Australasian Journal of Combinatorics} \textbf{69} (2017), 130--144.

 \bibitem{EGS}
N.\ Nasr Esfahani, I.\ Goldberg and D.R.\ Stinson.
Some results on the existence of $t$-all-or-nothing transforms over arbitrary alphabets.
\emph{IEEE Transactions on Information Theory} \textbf{64} (2018), 3136--3143.
 
 \bibitem{MacSl}
 F.J.\ MacWilliams and N.J.A.\ Sloane.
 {\em The Theory of Error-Correcting Codes}.
 North-Holland, 1977.
 
\bibitem{R}
R.L.\ Rivest.
All-or-nothing encryption and the package transform.
{\em Lecture Notes in Computer Science}
{\bf 1267} (1997),
 210--218 (Fast Software Encryption 1997).
 
\bibitem{St} 
D.R.\ Stinson.
Something about all or nothing (transforms).
\emph{Designs, Codes and Cryptography} \textbf{22} (2001), 133--138.


\bibitem{WCJ}
X.\ Wang, J.\ Cui and L.\ Ji.
Linear $(2, p, p)$-AONTs exist for all primes $p$.
\emph{Designs, Codes and Cryptography} {\bf 87} (2019), 2185--2197.


\bibitem{ZZWG}
Y.\ Zhang, T.\ Zhang, X.\ Wang and  G.\ Ge, 
Invertible binary matrices with maximum number of 2-by-2 invertible submatrices,
{\it Discrete Mathematics} {\bf 340} (2017) 201--208.


\end{thebibliography}
\end{document}